\documentclass[11pt]{amsart}
\usepackage{paper}
\usepackage[color=yellow!40]{todonotes}

\title[Norm Form Equations and Linear Divisibility Sequences]{Norm Form Equations and \\ Linear Divisibility Sequences}

\author{Elisa Bellah}

\begin{document}
\begin{abstract}
Finding integer solutions to norm form equations is a classical Diophantine problem. Using the units of the associated coefficient ring, we can produce sequences of solutions to these equations. It is known that these solutions can be written as tuples of linear recurrence sequences. We show that for certain families of norm forms defined over quartic fields, there exist integrally equivalent forms making any one fixed coordinate sequence a linear divisibility sequence.
\end{abstract}

\maketitle

\section{Introduction}
\label{sec: 1}
Let $K$ be a number field, and $W=\{w_1, \dots, w_n\}$ a $\Q$-linearly independent subset of $K$. The \textit{norm form} associated to the set $W$ is the rational form defined by
\begin{equation}
\label{normformdef} F_W(X_1, \dots, X_n):=N_K(X_1w_1+\cdots+X_nw_n).
\end{equation}
Given a norm form $F_W$, it is a classical Diophantine problem to ask for integer solutions to equations of the form
\begin{equation}\label{normform}
F_W(X_1, \dots, X_n)=c,
\end{equation}
where $c$ is a fixed nonzero integer. We consider two simple examples. 

\begin{example*} Let $D$ be a nonsquare integer. If we let $W = \{1, \sqrt{D}\}$, then the corresponding norm form defined over $\Q(\sqrt{D})$ is $F_W(X, Y) =X^2-DY^2$. In this case, we see that (\ref{normform}) is a Pell-type equation.
\end{example*}

\begin{example*} If $W$ is an integral basis for a number field $K$, the set of solutions to (\ref{normform}) with $c =\pm 1$ gives a complete list of units in $K$. So, the problem of finding units in a number field can be interpreted as such a Diophantine problem. 
\end{example*}

It is well known that integer solutions to (\ref{normform}) can be generated by tuples of linear recurrence sequences. In Theorems \ref{thm1} and \ref{thm2}, we consider norm forms defined over certain quartic fields, and show how to explicitly construct integrally equivalent forms so that these recurrence sequences are linear divisibility sequences. Before stating our main results, we first discuss the relevant background.\\

Given a $\Q$-linearly independent set $W$, let $M$ be the $\Z$-module in $K$ generated by $W$. Observe that if $T$ is another basis for $M$, the norm forms $F_W$ and $F_T$, which are defined in (\ref{normformdef}), are integrally equivalent. That is, there exist integers $a_{ij}$ so that if 
\[Y_i = \sum_{j=1}^n a_{ij} X_j,\]
then we have $F_W(X_1, \dots, X_n) = F_T(Y_1, \dots, Y_n)$. So, integer solutions to (\ref{normform}) can be found by instead studying the elements in the associated module $M$ of fixed norm $c$. Let 
\[\O_M:=\{\alpha \in K \mid \alpha M \subseteq M\}\] denote the \textit{coefficient ring} of the module $M$. 
It is well-known that when the module $M$ is full in $K$ (that is, when $\rank M =[K: \Q]$)  the set of elements in $M$ of fixed norm $c$ can be written as a disjoint union of finitely many families
\[\alpha_1\, \U_M^{+}, \dots, \alpha_{\ell}\, \U_M^{+},\]
where 
\begin{equation}
\label{positiveunitgroup} \U_M^+ := \{\epsilon \in \O_M \mid N_K(\epsilon)=1\}
\end{equation}
denotes the \textit{positive unit group} of $M$ (see Chapter 2 of \cite{borevich}, or Section 2 of \cite{schmidt}, for example). In \cite{schmidt}, Schmidt also gave a similar characterization in the case where $M$ is not full.\\

Let $\beta$ be an element in $M$ with $N_K(\beta)=c$, and suppose that $K \not=\Q$ and that $K$ is not imaginary quadratic. By Dirichlet's unit theorem, $\U_M^+$ is a finitely generated abelian group with positive rank. So, for any nontorsion element $\epsilon$ of $\U_M^+$, we can generate an infinite sequence of elements in $M$ of fixed norm $c$ given by 
\[
\alpha(k)=\beta\epsilon^k, \text{ where $k \in \Z_{\geq 0}$}.
\]
So, if we write
\begin{equation} \label{xidef}
\alpha(k)=x_1(k) w_1+\cdots+x_n(k) w_n,
\end{equation}
then we obtain infinitely many solutions $(x_1(k), \dots, x_n(k))$ to (\ref{normform}). Furthermore, the characterization above implies that all solutions to (\ref{normform}) are obtained in this way.

\begin{mydef} A linear recurrence sequence $b(k)$ is a \textit{linear divisibility sequence} (LDS) if $b(k)$ has the following property: for all $n, m \in \Z_{> 0}$,
\[ n \mid m \Rightarrow b(n) \mid b(m).\]
\end{mydef}

Divisibility sequences have been widely studied. Oftentimes, this extra structure is helpful in understanding further number theoretic properties of a given sequence. For example, every Lucas sequence is a LDS. This property was used in \cite{BHV} to study the primitive divisors of Lucas sequences, and in \cite{smyth} to study their index divisibility sets, as well as in many other results throughout the literature. Elliptic Divisibility Sequences, introduced by Ward in \cite{ward}, are examples of nonlinear divisibility sequences. Similar results for these sequences have also been found, such as in \cite{stangesilverman} and \cite{EDSprimdivs}.\\

In this paper, we show that for certain families of norm equations defined over quartic fields, we can find an integrally equivalent form so that one of the sequences $\{x_i(k): k \in \Z_{\geq 0}\}$ defined in (\ref{xidef}) is a linear divisibility sequence. The families we consider are motivated by the following theorem of Kubota.

\begin{prop}[\cite{kubota}, Theorem 1] \label{kubota} Let $K$ be a real biquadratic field with quadratic subfields $L_i$, and let $\epsilon_i$ be a fundamental unit of $L_i$. Then, $K$ has a system of fundamental units of one of the following forms, up to relabeling:
\begin{enumerate}
\item[(i)] $\epsilon_1, \epsilon_2, \epsilon_3$
\item[(ii)] $\sqrt{\epsilon_1}, \epsilon_2, \epsilon_3$
\item[(iii)] $\sqrt{\epsilon_1}, \sqrt{\epsilon_2}, \epsilon_3$
\item[(iv)] $\sqrt{\epsilon_1\epsilon_2}, \epsilon_3, \epsilon_3$
\item[(v)] $\sqrt{\epsilon_1\epsilon_2}, \sqrt{\epsilon_3}, \epsilon_2$
\item[(vi)] $\sqrt{\epsilon_1\epsilon_2}, \sqrt{\epsilon_2\epsilon_3}, \sqrt{\epsilon_3\epsilon_1}$
\item[(vii)] $\sqrt{\epsilon_1\epsilon_2\epsilon_3}, \epsilon_2, \epsilon_3$
\item[(viii)] $\sqrt{\epsilon_1\epsilon_2\epsilon_3}, \epsilon_2, \epsilon_3$, with $N_{K_1}(\epsilon_i) = -1$ for $i = 1, 2, 3$,
\end{enumerate}
where $\sqrt{\epsilon}$ denotes any element $\eta \in K$ with $\eta^2 = \epsilon$.
Furthermore, there are infinitely many $K$ of each type.
\end{prop}

Proposition \ref{kubota} tells us that to study solutions to a norm form equation \[F_W(X_1, X_2, X_3, X_4)=c\] defined over a real biquadratic field, it suffices to understand the coordinates of the sequences $\alpha(k) = \beta \eta^k$ where $\eta$ is of one of the following three types:
\begin{enumerate}
\item[(a)] $\eta$ is a unit in quadratic subfield of $K$,
\item[(b)] $\eta^2$ is a unit in a quadratic subfield of $K$, or
\item[(c)] $\eta$ is a product of units of types $(a)$ and $(b)$
\end{enumerate}

Our main results concern the sequences $\alpha(k) = \beta \eta^k$ where $\eta$ is type (b). In fact, our results hold for quartic fields containing a unit of type (b) more generally. We will show the following. 

\begin{thm} \label{thm1} Let $K$ be a quartic field with a real quadratic subfield $L$ containing a quartic unit $\eta$ of positive norm, so that $\eta^2$ is a unit in $L$. Fix an element $\beta \in K$, and write $\alpha(k)=\beta \eta^k$. Then, there is a choice of basis  $W=\{w_1, w_2, w_3, w_4\}$ for the module $M'=\beta \,\Z[\eta]$, which we construct explicitly, so that if we write
\[\alpha(k)=x_1(k)w_1+\cdots+x_4(k)w_4\]
then $\{x_1(k): k \in \Z_{\geq 0}\}$ is a LDS.
\end{thm}

\begin{thm}\label{thm2}   Let $M=\Z[\sqrt{m}, \sqrt{m+1}]$, where $m$ and $m+1$ are non-square integers. Then, $\eta=\sqrt{m}+\sqrt{m+1}$ is a unit in the positive unit group $\U_M^+$ with $\eta^2$ a unit in a quadratic subfield of $K=\Q(\eta)$, and there is a choice of basis $W=\{w_1, w_2, w_3, w_4\}$ for the module $M$, which we construct explicitly, so that if we write
\[\eta^k=x_1(k)w_1+\cdots+x_4(k)w_4,\]
then $\{x_1(k): k \in \Z_{\geq 0}\}$ is a LDS. 
\end{thm}

\begin{rmk} Note that Theorems \ref{thm1} and \ref{thm2} hold for the sequence \[\{x_i(k): k \in \Z_{\geq 0}\}\] for any fixed $i \in \{1, 2, 3, 4\}$, just by changing the basis to reindex our coordinates. However, we show in Sections \ref{sec: 3} and \ref{sec: 4} that there does not exist a choice of basis for the modules $M'$ and $M$ in Theorems \ref{thm1} and \ref{thm2} so that the coordinate sequences $x_1(k), x_2(k), x_3(k), x_4(k)$ defined in (\ref{xidef}) are LDS simultaneously. 
\end{rmk}

This paper is organized as follows. In Section \ref{sec: 2}, we show that the sequences $\{x_i(k): k \in \Z_{\geq 0}\}$ defined in (\ref{xidef}) are linear recurrence sequences, each with characteristic polynomial equal to the minimal polynomial of our unit $\epsilon$. In Section \ref{sec: 3}, we provide some background on Lucas sequences, and discuss how this can be used to study norm forms defined over real quadratic fields. In Section \ref{sec: 4}, we prove Theorems \ref{thm1} and \ref{thm2}. In Section \ref{sec: 5}, we discuss a related sequence proposed by Silverman in \cite{algints}, and provide examples where Conjecture 9 of this paper holds.

\section{Coordinate Sequences}
\label{sec: 2}
Let $M$ be a full module in a number field $K$, and $\epsilon$ a nontorsion element in the positive unit group $\U_M^+$ defined in (\ref{positiveunitgroup}). For $\beta \in M$ with $N_K(\beta)=c$, set $\alpha(k)=\beta \epsilon^k$. If we choose a basis $W=\{w_1, \dots, w_n\}$ for $M$, and write
\begin{equation}
\label{xidef2} \alpha(k)=x_1(k)w_1+\cdots+x_n(k)w_n,
\end{equation}
then we obtain tuples of solutions $(x_1(k), \dots, x_n(k))$ to the corresponding norm form equation $F_W(X_1, \dots, X_n)=c$. 

\begin{mydef} We call the integer sequences $\{x_i(k) : k \in \Z_{\geq 0}\}$, where $x_i(k)$ is defined in (\ref{xidef2}), the \textit{coordinate sequences} of $\alpha(k)$ with respect to our choice of basis $W$. 
\end{mydef}

Let $b(k)$ be an integer sequence satisfying the linear homogeneous recurrence
\begin{equation} \label{recurrence} b(k+d)=s_1b(k+d-1)+\cdots+s_db(k), \end{equation}
where $s_i \in \Z$ and $d \in \Z_{\geq 0}$ are fixed. Then, the \textit{characteristic polynomial} for this recurrence is given by 
\[f(X)=X^d-s_1X^{d-1}-\cdots-s_d.\]
When recurrence (\ref{recurrence}) is of minimal order (that is, when $d$ is minimal), $f(X)$ is called the \textit{minimal polynomial} of the sequence $b(k)$. \\

In this section, we show that the coordinate sequences $\{x_i(k): k \in \Z_{\geq 0}\}$ have characteristic polynomial equal to the minimal polynomial of $\epsilon$. We also provide sufficient conditions so that the minimal polynomial of the sequence $\{x_i(k): k \in \Z_{\geq 0}\}$ is equal to the minimal polynomial of $\epsilon$. 

\begin{prop} \label{prop1} Let $K$ be a number field, and take elements $\gamma, \theta \in K$. Consider the sequence
$x(k)=\Tr_{K/ \Q}(\gamma \theta^k).$
\begin{enumerate}
\item[(a)] The sequence $x(k)$ satisfies a linear homogeneous recurrence with characteristic polynomial equal to the minimal polynomial of $\theta$.
\item[(b)] Let $L=\Q(\theta)$. If $\Tr_{K/L}(\gamma) \not=0$, then the minimal polynomial of the sequence $x(k)$ is equal to the minimal polynomial of $\theta$. 
\end{enumerate}
\end{prop}

\begin{rmk} Suppose that $\theta$ has minimal polynomial 
\[f(X)= X^d -s_1 X^{d-1} - \cdots - s_d.\]
Then, Proposition \ref{prop1}(a) implies that the sequence $x(k) = \Tr_{K/ \Q}(\gamma \theta^k)$ satisfies the recurrence
\[x(k+d) = s_1 x(k+d-1) +\cdots + s_d x(k).\]
However, it is possible that this recurrence is not minimal. For example, take $K = \Q(\sqrt{2}, \sqrt{3}, \sqrt{5})$, $\theta = \sqrt{2}+\sqrt{3}$ and $\gamma = \sqrt{5}$. Then, Proposition \ref{prop1} (a) implies that $x(k)$ satisfies an order 4 recurrence, but we can check that $x(k) = 0$ for $k = 0, 1, 2, 3$. So, $x(k)$ is a constant sequence, while $\deg \theta = 4$. \\

There does not appear to be a complete characterization for when the sequence $x(k)$ is exactly of order $\deg \theta$ in the current literature, so Proposition \ref{prop1} (b) gives a new result in this direction. We note that Proposition \ref{prop1} (a) follows from known results on generalized power sums (see Chapter 1 of \cite{recurrence}, for example), but we provide a more elementary proof below.
\end{rmk}

\begin{proof}[Proof of Proposition \ref{prop1}] Let $\tilde{K}$ denote the Galois closure of $K$, and let 
\[\Gal(\tilde{K}/ \Q) = \{\sigma_1, \dots, \sigma_n\}.\] 
Set $\gamma_i:=\sigma_i(\gamma)$ and $\theta_i:=\sigma_i(\theta)$, for $i \in \{1, \dots, n\}$. Then, we can write
\begin{equation}
\label{aneq} x(k)=\Tr_{K/ \Q}(\gamma \theta^k)=\sum_{i=1}^n \gamma_i \theta_i^k.
\end{equation}
Let 
\[f(X)=X^d -s_1X^{d-1} - \cdots - s_d\]
be the minimal polynomial of $\theta$ over $\Q$. Then,
\begin{align*}
\sum_{j=1}^d s_j x(k+d-j)	&= \sum_{j=1}^d \sum_{i=1}^n s_j \gamma_i \theta_i^{k+d-j} \,\,\, \text{ by (\ref{aneq})}\\
					&=\sum_{i=1}^n \gamma_i \theta_i^k \sum_{j=1}^d s_j \theta_i^{d-j} \\
					&=\sum_{i=1}^n \gamma_i \theta_i^k \theta_i^d,
\end{align*}
where the final equality follows because each $\theta_i$ is a root of $f(X)$. So, our sequence satisfies the recurrence $x(k+d)=\sum_{j=1}^d s_jx(k+d-j),$ which has characteristic polynomial equal to $f(X)$, which proves part (a). \\

Next, suppose that $x(k)$ satisfies an order $m$ recurrence for $0 < m \leq d$, say
\[x(k+m)= \sum_{j=1}^{m} r_j x(k+m-j),\]
where $r_j \in \Z$. Then, we have
\[\Tr_{K/ \Q}(\gamma \theta^{k+m})=\sum_{j=1}^m r_j \Tr_{K/ \Q}(\gamma \theta^{k+m-j}),\]
and by linearity of the trace, we get $\Tr_{K/ \Q}(C \theta^k \cdot \gamma)=0,$
where \[C=\theta^m-\sum_{i=1}^m r_i \theta^{m-i}.\] 
Order the embeddings so that $\sigma_1(\theta)=\theta_1, \dots, \sigma_d(\theta)=\theta_d$ are distinct, and $\sigma_i(\theta) = \sigma_{dm+i}(\theta)$ for $m =1, 2, \dots, \ell$. Then,
\begin{align*}
\Tr_{K/ \Q}(C \theta^k \cdot \gamma) &=\sigma_1(C\theta^k) \left(\sigma_1(\gamma) + \sigma_{d+1}(\gamma) +\cdots + \sigma_{\ell d + 1} (\gamma)\right) \\
&+ \sigma_2(C\theta^k)\left(\sigma_2(\gamma) + \sigma_{d+2}(\gamma)+\cdots + \sigma_{\ell d + 2} (\gamma) \right)\\
&\,\,\, \vdots \\
&+ \sigma_d(C\theta^k) \left(\sigma_d(\gamma) + \sigma_{2d}(\gamma) + \cdots + \sigma_{(\ell+1)d} (\gamma) \right)
\end{align*}
where $n = (\ell+1) d$. For $i =1, \dots, d$. Set
\[ S_i = \sigma_i(\gamma) + \sigma_{d+i}(\gamma) + \cdots + \sigma_{\ell d + i}(\gamma).\]
 Then, we can write
\begin{equation}
\label{matrices}
\begin{pmatrix} \sigma_1(C \theta^0) & \cdots & \sigma_d(C \theta^0) \\ \vdots & \ddots & \vdots \\ \sigma_1(C \theta^{d-1}) & \cdots & \sigma_d(C \theta^{d-1}) \end{pmatrix} 
\begin{pmatrix} S_1 \\ \vdots \\ S_d \end{pmatrix}= \begin{pmatrix} 0 \\ \vdots \\ 0 \end{pmatrix},\end{equation}
where $L_i=\Q(\theta_i)$ and $K_i=\sigma_i(K)$. Without loss of generality, suppose that $\sigma_1(\theta) = \theta$. Then, 
\[ S_1 = \Tr_{K / L}(\gamma), \]
where $L = \Q(\theta)$. If $C \not=0$ then the set $\{C, C\theta, \dots, C\theta^{d-1}\}$ would be $\Q$-linearly independent, and we would have
\[\det \begin{pmatrix} \sigma_1(C \theta^0) & \cdots & \sigma_n(C \theta^0) \\ \vdots & \ddots & \vdots \\ \sigma_1(C \theta^{d-1}) & \cdots & \sigma_d(C \theta^{d-1}) \end{pmatrix} =\disc(C, \theta, \dots, C\theta^{d-1})^{1/2} \not=0.\]
But this contradicts (\ref{matrices}), since $S_1 \not =0$ by assumption. So, we must have $C=0$, and so $\theta$ is a root of 
\[X^m-\sum_{i=1}^m r_i X^{m-i} \in \Z[X]\]
But since $\theta$ is degree $d$, and $m \leq d$ we get $m=d$. Hence, the recurrence
\[x(k+d)=\sum_{j=1}^d s_j x(k+d-j)\]
is minimal, and so $f(X)$ is the minimal polynomial of the sequence $x(k)$.
\end{proof} 

We have the following Corollary to Proposition \ref{prop1}.

\begin{cor} \label{cor1} Let $K$ be a number field and $M$ a full module in $K$. Suppose that $\epsilon$ is a nontorsion element in $\,\U_M^+$. For a fixed $\beta \in M$, let \[\alpha(k)=\beta \epsilon^k\] and $\{x_i(k): k \in \Z_{\geq 0}\}$ be a coordinate sequence of $\alpha(k)$ with respect to some basis, as defined in (\ref{xidef2}). Then, $\{x_i(k): k \in \Z_{\geq 0}\}$ is a linear recurrence sequence with characteristic polynomial equal to the minimal polynomial of $\epsilon$. Furthermore, if $\deg\epsilon=[K: \Q]$ then the minimal polynomial of this sequence is equal to the minimal polynomial of $\epsilon$.
\begin{proof}
Let $W=\{w_1, \dots, w_n\}$ be any basis for $M$, and write
\[\alpha(k)=x_1(k)w_1+\cdots+x_n(k)w_n.\]
Since $M$ is a full module, $W$ is a $\Q$-basis for $K$. So, there exists a dual basis $W^*=\{w_1^*, \dots, w_n^*\}$ to $W$ with respect to the trace pairing. That is, $W^*$ is a basis for $K$, and we have \[\Tr_{K/ \Q}(w_i^* w_j)=\delta_{ij}\]
for all $i, j$, where \[\delta_{ij} = \begin{cases} 1 & \text{ if } i =j \\ 0 & \text{ if } i \not=j. \end{cases}\] Let $\gamma=w_i^* \beta$. Then we have \[x_i(k)=\Tr_{K/ \Q}(\gamma \epsilon^k).\] Note that if $\deg \epsilon=[K: \Q]$ then $\Q(\epsilon)=K$. So \[\Tr_{\Q(\theta)}^K(\gamma)=\gamma\not=0.\] 
Hence, the result follows from Proposition \ref{prop1} (b).
\end{proof}
\end{cor}


\section{Norm Form Equations over Real Quadratic Fields}
\label{sec: 3}
Suppose that $K$ is a real quadratic field, and let $M$ be a full module in $K$. For any $\beta \in M$ and $\epsilon$ a nontorsion element in $\,\U_M^+$, let $\alpha(k)=\beta \epsilon^k$ as before. Since $\epsilon$ is degree 2 over $\Q$, Corollary \ref{cor1} implies that the coordinate sequences of $\alpha(k)$ are order 2 linear homogeneous recurrence sequences. Such sequences have been well-studied, and so Corollary \ref{cor1} implies some immediate consequences. First, we provide the relevant background.\\

Let $P, Q$ be nonzero coprime integers. Then, the \textit{Lucas sequence} with integer parameters $(P, Q)$ is the order 2 linear recurrence sequence $u_k$ with initial values $u_0=0$, $u_1=1,$ and recurrence
\[u_{k+2}=Pu_{k+1}-Qu_k.\]
For example, the Fibonacci sequence is the Lucas sequence with integer parameters $(1, -1)$. Let $\theta, \bar{\theta}$ be roots of the polynomial $X^2-PX-Q$. It is a short exercise to show that the terms of the Lucas sequence with integer parameters $(P, Q)$ satisfies the explicit formula
\[u_k=\frac{\theta^k-\bar{\theta}^k}{\theta-\bar{\theta}}.\]
Note that Lucas sequences are sometimes defined by the parameters $(\theta, \bar{\theta})$, rather than the integer parameters $(P, Q)$.\\

The following elementary Lemma is well-known. We provide a short proof for completeness. 
\begin{lem}\label{LucasDivisibility} Every Lucas sequence is a LDS. 
\begin{proof}
Let $P, Q$ be nonzero coprime integers, and consider the matrix
\[A=\begin{pmatrix} P & -Q \\ 1 & 0 \end{pmatrix}.\]
Observe that for any positive integer $k$, we have
\[A^k=\begin{pmatrix} u_{k+1} & -Q u_k \\ u_k & -Q u_{k-1} \end{pmatrix},\]
where $u_k$ is the Lucas sequence with integer parameters $(P, Q)$. Now, take any positive integers $m, n$. Then we have
\[A^{mn}=\begin{pmatrix} u_{m+1} & -Q u_m \\ u_m & -Qu_{m-1} \end{pmatrix}^n \equiv \begin{pmatrix} * & * \\ 0 & * \end{pmatrix} (\mod u_m).\]
On the other hand, we have
\[A^{mn}=\begin{pmatrix} u_{mn+1}  & -Q u_{mn} \\ u_{mn} & -Q u_{mn-1} \end{pmatrix}.\]
Comparing the lower left hand entries, we see that $u_m \mid u_{mn}$
for every $m, n \in \Z_{>0}$. So, $u_k$ is a LDS. 
\end{proof}
\end{lem}

In fact, an order 2 linear recurrence sequence $a(k)$ is a LDS if and only if $a(k)=c \cdot u_k$, where $c$ is a nonzero constant, and $u_k$ is any Lucas sequence (see Theorem 1 of \cite{he}, for example). So, given $\alpha(k)$ as in \ref{mainprop}, it is not possible to find a basis for $M$ so that $x_1(k)$ and $x_2(k)$ are LDS simultaneously.  \\

We have the following Proposition.

\begin{prop} \label{mainprop} Let $K$ be a real quadratic field and $M$ a full module in $K$. Fix an element $\beta \in M$ and write $\alpha(k)=\beta \epsilon^k$. Then, there is a choice of basis $W=\{w_1, w_2\}$ for $M$, which we construct explicitly, so that if we write
$\alpha(k)=x_1(k)w_1+x_2(k)w_2$
then the sequence $x_1(k)$ is a LDS. 
\end{prop}

\begin{proof} By Lemma \ref{LucasDivisibility}, it suffices to find a basis $\{w_1, w_2\}$ for $M$ so that $x_1(0)=0$. Let $\{t_1, t_2\}$ be any basis for $M$, and $B$ be the matrix given by
\[\begin{pmatrix} \beta \\ \beta \epsilon \end{pmatrix}=B \begin{pmatrix} t_1 \\ t_2\end{pmatrix}.\]
Note that $\exists C \in \GL_2(\Z)$ so that $BC$ is lower triangular. So, we can define a new basis $\{v_1, v_2\}$ from $\{t_1, t_2\}$ by change of basis matrix $C^{-1}$. Then,
\begin{equation}\label{quadmatrix1}
 \begin{pmatrix} \beta \\ \beta \epsilon \end{pmatrix}=\begin{pmatrix} a_{11} & 0 \\ a_{21} & a_{22} \end{pmatrix} \begin{pmatrix} v_1 \\ v_2 \end{pmatrix},
\end{equation}
for some $a_{ij} \in \Z$. Now, let $W=\{w_1, w_2\}$ be the basis defined by
\[\begin{pmatrix} w_1 \\ w_2 \end{pmatrix} =\begin{pmatrix} 1 & 1 \\ 1 & 0 \end{pmatrix} \begin{pmatrix} v_1 \\ v_2\end{pmatrix}.\]
We claim that we can take $W$ as our desired basis. To see this, observe that
\[ \begin{pmatrix} 0 & a_{11} \\ a_{22} & a_{21}-a_{22} \end{pmatrix} \begin{pmatrix} w_1 \\ w_2 \end{pmatrix} =
\begin{pmatrix} a_{11} & 0 \\ a_{21} & a_{22} \end{pmatrix} \begin{pmatrix} v_1 \\ v_2 \end{pmatrix}.\]
So, if we write $\alpha(k)=x_1(k)w_1+x_2(k)w_2$ then by (\ref{quadmatrix1}) $x_1(k)$ has initial conditions $x_1(0)=0$ and $x_1(1)=a_{22}$. So, $x_1(k)=a_{22} u_k$, where $u_k$ is the Lucas sequence with parameters $(\epsilon, \bar{\epsilon})$. By Corollay \ref{cor1} we know that $x_1(k)$ is an order 2 recurrence sequence, and so we must have $a_{22}\not=0$. Hence, $x_1(k)$ is a LDS. 
\end{proof}

\section{Norm Form Equations over Quartic Fields}
\label{sec: 4}
Let $K$ be a quartic field, and $M$ a full module in $K$. Choose a basis $\beta \in M$, and suppose there exists a unit $\eta \in \U_M^+$ of degree 4 over $\Q$. By Corollary \ref{cor1}, the coordinate sequences of $\alpha(k)=\beta \eta^k$ are order 4 linear recurrence sequences. Unlike in the order 2 case, much less is known about higher-order linear recurrence sequences, and so it is generally quite challenging to determine when an arbitrary order 4 linear recurrence sequence is a LDS. \\

Suppose that $\eta$ is a quartic unit with $\eta^2 = : \epsilon$ a unit in a quadratic subfield of $\Q(\eta)$. 
Recall, by Proposition \ref{kubota}, this is one of the three cases needed to understand solutions to norm forms over real biquadratic fields.  Let $\tilde{K}$ be the Galois closure of $K$. Observe that for $\sigma \in \Gal(\tilde{K}/ \Q)$ we have $\sigma(\epsilon) = \sigma(\eta)^2.$
So, the conjugates of $\eta$ are of the form 
\begin{equation} \label{conjugates}
\pm \sqrt{\epsilon}, \pm \sqrt{\bar{\epsilon}},
\end{equation} 
where $\bar{\epsilon}$ denotes the conjugate of $\epsilon$.
Since $\eta$ is of degree 4 over $\Q$, it has minimal polynomial 
\begin{equation} \label{quarticpoly}
f(X)=X^4-(\epsilon+\bar{\epsilon})X^2+1.
\end{equation}
So, Corollary \ref{cor1} implies that the coordinate sequences $x(k)$ of $\alpha(k)$ are order 4 linear recurrence sequences satisfying
\begin{equation}
\label{order4rec} x(k+4)=T x(k+2)-x(k),
\end{equation}
where $T = \epsilon+\bar{\epsilon}$. 
The following Proposition gives sufficient initial conditions for $x(k)$ to be a LDS, and will be used to prove our main results. 

\begin{prop} \label{prop} Let $x(k)$ be an order 4 linear recurrence sequence with initial conditions
$x(0)=0, \, x(1)=x(2)=a, \, x(3)=a(T+1),$
and recurrence $x(k+4)=Tx(k+2)-x(k)$, where $a$ and $T$ are nonzero integers. Then, $x(k)$ is a LDS. 
\begin{proof} Note that it suffices prove our claim for $a=1$. 
Let $u_k$ denote the Lucas sequence with integer parameters $(T, 1)$. Since we assumed that $x(0)=0$ and $x(2)=1$, we have $x(2n)=u_n$ for every $n \in \Z_{\geq 0}$. Consider the matrix
\[A=\begin{pmatrix} T & -1 \\ 1 & 0 \end{pmatrix}.\]
Recall from the proof of Lemma \ref{LucasDivisibility} that we have the identity
\begin{equation} \label{identity0} A^n=\begin{pmatrix} u_{n+1}	& - u_n \\ u_n & -u_{n-1} \end{pmatrix},
\end{equation}
and so we have
\begin{equation} \label{identity1}
A^n= \begin{pmatrix} x(2n+2) & -x(2n) \\ x(2n) & -x(2n-2) \end{pmatrix},
\end{equation}
for every $n \in \Z_{>0}$. Using the recurrence for $x(k)$, we observe that
\begin{equation} \label{identity2}
A^n \begin{pmatrix} x(3) \\ x(1) \end{pmatrix}=\begin{pmatrix} x(2n+3) \\ x(2n+1) \end{pmatrix}.
\end{equation}
Combining (\ref{identity1}) and (\ref{identity2}) yields
\[\begin{pmatrix} x(2n+3) \\ x(2n+1) \end{pmatrix} = \begin{pmatrix} x(3)x(2n+2) -x(1)x(2n) \\ x(3)x(2n)-x(1)x(2n-2) \end{pmatrix}.\]
That is, we have $x(2n+1)=x(3)x(2n)-x(1)x(2(n-1))$ for any positive integer $n$. Recalling that $x(1)=1,$ $x(3)=T+1$ and $x(2n)=u_n$, we obtain
\begin{align*}
x(2n+1)	&=(T+1)u_n-u_{n-1}\\
		& =u_{n+1} +u_n,
\end{align*}
where the final equality follows by using the recurrence for $u_k$. So, we have
\begin{equation}
\label{x1}
x(k)=\begin{cases} u_n, & \text{ if } k=2n \\ u_{n+1}+u_{n}, & \text{ if } k=2n+1, \end{cases}
\end{equation}
for any $k \in \Z_{\geq 0}$. 
Note that we need to show $x(k) \mid x(k\ell)$ for every $k, \ell \in \Z_{\geq 0}$. Suppose that $k=2n$. Then, $x(k)=u_n$ and $x(k\ell)=u_{n\ell}$. So, by Lemma \ref{LucasDivisibility} we have $x(k) \mid x(k \ell)$. Next, suppose that $k=2n+1$ and $\ell=2m$.
Noting that $A^{2n}=(A^n)^2$, and using identity (\ref{identity0}) we have
\[\begin{pmatrix} u_{2n+1} & -u_{2n} \\ u_{2n} & - u_{2n-1} \end{pmatrix}=\begin{pmatrix} u_{n+1} & -u_n \\ u_n & -u_{n-1} \end{pmatrix}^2.\]
After squaring the matrix on the right, we compare the upper left-hand entries to get the identity
$
u_{2n+1}=u^2_{n+1}-u_n^2.
$
So, we have
\begin{align*}
\frac{x(2k)}{x(k)}	&=\frac{x(2(2n+1))}{x(2n+1)}\\
			&=\frac{u_{2n+1}}{u_{n+1}+u_n}\\
			&=u_{n+1}- u_{n} \in \Z.
\end{align*}
Hence, $x(k) \mid x(2k)$, and by the previous case we have
\[x(2k) \mid x(2km) \Rightarrow x(k) \mid x(k\ell).\] 
Now, suppose that $k=2n+1$ and $\ell=2m+1$. Let $\epsilon, \bar{\epsilon}$ denote the roots of $X^2-TX+1$. Recall from Section 3 that we can write
\[u_k=\frac{\epsilon^k - \bar{\epsilon}^k}{\epsilon-\bar{\epsilon}},\]
for every $k \in \Z_{\geq 0}$. So, we have
\begin{align*}
x(2n+1)	& =  u_{n+1}+u_{n}\\
		&= \frac{\epsilon^{n+1}-\bar{\epsilon}^{n+1}}{\epsilon -\bar{\epsilon}} + \dfrac{\epsilon^n - \bar{\epsilon}^n}{\epsilon-\bar{\epsilon}}\\
		&= \frac{\epsilon^n(\epsilon+1)-\bar{\epsilon}^n(\bar{\epsilon}+1)}{\epsilon-\bar{\epsilon}}\\
		&= \frac{\epsilon^n(\epsilon+1)-\dfrac{1}{\epsilon^{n+1}} (1+\epsilon)}{\epsilon-\bar{\epsilon}}\\
		&=\frac{\epsilon+1}{\epsilon-\bar{\epsilon}} \cdot \frac{\epsilon^{2n+1}-1}{\epsilon^{n+1}}.
\end{align*}

This gives
\begin{align*}
\frac{x((2n+1)(2m+1))}{x(2n+1)}	&= \frac{x(2(2nm+n+m)+1)}{x(2n+1)}\\
						&= \frac{\epsilon^{2(2nm+n+m)+1}-1}{\epsilon^{2nm+n+m+1}} \cdot \frac{\epsilon^{n+1}}{\epsilon^{2n+1}-1}\\
						&=\frac{\epsilon^{(2n+1)(2m+1)}-1}{\epsilon^{2n+1}-1} \cdot \frac{1}{\epsilon^{m(2n+1)}}.
\end{align*}
To see this value is in $\Z$, let $\alpha=\epsilon^{2n+1}$. Then, from above we obtain
\begin{align*}
\frac{x((2n+1)(2m+1))}{x(2n+1)} &= \frac{\alpha^{2m+1}-1}{\alpha-1} \cdot \frac{1}{\alpha^m}\\
					&=\frac{\alpha^{2m}+\alpha^{2m-1}+\cdots+\alpha+1}{\alpha^m}\\
					&= (\alpha^m+\alpha^{-m})+\cdots+(\alpha+\alpha^{-1})+1.
\end{align*}
Since $\alpha=\epsilon^{2n+1}$ and $N_K(\epsilon)=1$, then $\alpha$ and $\alpha^{-1}$ are quadratic conjugates. So, we have
$\alpha^t+\alpha^{-t} \in \Z$
for every $t=1, \dots, m$. Hence, \[x(2n+1) \mid x((2n+1)(2m+1)),\] and so $x(k)$ is a LDS.
\end{proof}
\end{prop}

Theorem \ref{thm1} will now follow from Proposition \ref{prop1}. Recall that $K$ is a quartic number field containing a quartic unit $\eta$ of positive norm so that $\eta^2$ is a unit in a quadratic subfield of $K$.

\subsection*{Proof of Theorem \ref{thm1}} Note that the module $M'=\beta\Z[\eta]$ has basis $\{\beta, \beta \eta, \beta \eta^2, \beta \eta^3\}$. Define the set $W=\{w_1, w_2, w_3, w_4\}$ by
\[\underbrace{\begin{pmatrix} 0 & 0 & 1 & 0 \\ 1 & 0 & 0 & 1 \\ 1 & 0 & 0 & 0 \\ T+1 & 1 & 0 & 0 \end{pmatrix}}_{A} \begin{pmatrix} w_1 \\ w_2 \\ w_3 \\ w_4 \end{pmatrix} = \begin{pmatrix} \beta \\ \beta \eta \\ \beta \eta^2 \\ \beta \eta^3 \end{pmatrix},\]
where $T = \epsilon+ \bar{\epsilon}$. 
Note that $A \in \GL_4(\Z)$, and so $W$ is a basis for $M$. 
Since $\eta$ has minimal polynomial \[f(X)=X^4-TX^2+1,\] 
then by Corollary \ref{cor1} we know that the sequence $x_1(k)$ is an order 4 linear recurrence sequence satisfying (\ref{order4rec}). Moreover, if we write $\alpha(k)$ in terms of the basis $W$, then 
\[x_1(0)=0, \,x_1(1)=x_2(1)=1, \text{ and } x_3(1)=\Tr_{K/ \Q}(\epsilon)+1.\]
So, by Proposition \ref{prop}, $x_1(k)$ is a LDS. \hfill $\qed$ \\

In the following Corollary, we provide explicit formulas for the coordinate sequences of $\alpha(k)$, with respect to the basis constructed in Theorem \ref{thm1}, in terms of Lucas sequences.

\begin{cor} Let $W=\{w_1, w_2, w_3, w_4\}$ be the basis for the module $\beta \Z[\eta]$ constructed in Theorem \ref{thm1}, and $\alpha(k)=\beta \eta^k$ be as above. If we write
\[\alpha(k)=x_1(k)w_1+\cdots+x_4(k)w_4,\]
then for any integer $k \geq 2$ we have\\
\[x_1(k) =\begin{cases} u_n & \text{ if } k=2n \\ u_{n+1}+u_n & \text{ if } k=2n+1, \end{cases}
\hspace{2em} x_2(k)=\begin{cases} 0 & \text{ if } k =2n \\ u_n & \text{ if } k=2n+1,\end{cases}\]
\vspace{0.2em}
\[\hspace{0.75em} x_3(k)=\begin{cases} - u_{n-1}  & \text{ if } k=2n \\ 0 & \text{ if } k=2n+1,\end{cases} 
\hspace{3.5em} x_4(k)=\begin{cases} 0 & \text{ if } k=2n \\ - u_{n-1} & \text{ if } k=2n+1, \end{cases}\]
\vspace{0.25em}

where $u_n$ is the Lucas sequence with parameters $(\epsilon, \bar{\epsilon})$, defined in Section \ref{sec: 2}. 

\begin{proof} Let $W$ be the basis constructed in the proof of Theorem \ref{thm1}, and write
\[\alpha(k)=x_1(k)w_1+\cdots+x_4(k)w_4.\]
Recall, by Corollary \ref{cor1} we know that all of the coordinate sequences $x_i(k)$ of $\alpha(k)$ satisfy the order 4 recurrence
\begin{equation}
\label{recurrenceagain}x_i(k+4)=T x_i(k+2)-x_i(k),
\end{equation}
where $T = \epsilon+\bar{\epsilon}$,
and by construction of our basis $W$ these sequences have initial conditions
\begin{center}
\begin{tabular}{c | c c c c}
$k$ 	& $x_1(k)$ 				& $x_2(k)$ 	& $x_3(k)$ 		& $x_4(k)$ 		\\\hline
0	& $0$					& $0$		& $1$			& $0$ 			\\
1	& $1$					& $0$		& $0$			& $1$ 			\\
2	& $1$					& $0$		& $0$			& $0$ 			\\
3	& $T+1$		& $1$		& $0$			& $0$
\end{tabular}
\end{center}

Let $\sigma_1, \dots, \sigma_4$ be the distinct embeddings $K \hookrightarrow \C$ fixing $\Q$, and let 
\[W^*=\{w_1^*, w_2^*, w_3^*, w_4^*\}\]
be a dual basis to $W$ with respect to the trace pairing on $K$. Recall from the proof of Corollary \ref{cor1} that we can write $x_i(k)=\Tr_{K/ \Q}(w_i^* \beta \eta^k)$. Also recall from (\ref{conjugates}) that the conjugates of $\eta$ are given by $\pm \sqrt{\epsilon}, \pm \sqrt{\bar{\epsilon}}$.  So, up to relabeling of the embeddings $\sigma_i$, we have
\begin{equation}
\label{explicit}
x_i(k)=(\gamma_{i1} +(-1)^k\gamma_{i2}) \sqrt{\epsilon}^k+(\gamma_{i3}+(-1)^k\gamma_{i4}) \sqrt{\bar{\epsilon}}^{\,k},
\end{equation}
for every $k \in \Z_{\geq 0}$, where $\gamma_{ij}=\sigma_j(w_i^*\beta)$. \\

From (\ref{x1}) in the proof of Proposition \ref{prop} we see that $x_1(k)$ satisfies the desired formula. Next, since $x_2(0)=x_2(2)=0$, then using the recurrence for $x_2(k)$ in (\ref{recurrenceagain}), we see that $x_2(2n)=0$ for every $n \in \Z_{\geq 0}$. From (\ref{explicit}), we have
\[x_2(2n+1)=(\gamma_{21}-\gamma_{22})\sqrt{\epsilon}^{2n+1}+(\gamma_{23}-\gamma_{24})\sqrt{\bar{\epsilon}}^{\, 2n+1}.\]
Since $x_2(1)=0$ and $N_L(\epsilon)=\epsilon \bar{\epsilon}=1$, we get
$\gamma_{23}-\gamma_{24}= -(\gamma_{21}-\gamma_{22}) \epsilon$. So,
\[x_2(2n+1)=(\gamma_{21} - \gamma_{22}) \left(\sqrt{\epsilon}^{2n+1} - \sqrt{\bar{\epsilon}}^{\, 2n-1}\right).\]
Using the equality above and the fact that $x_2(3)=1$, we have
\[\gamma_1-\gamma_2=\dfrac{1}{\sqrt{\epsilon}^3-\sqrt{\bar{\epsilon}}}\]
which implies that
\[x_2(2n+1)	=\frac{\sqrt{\epsilon}^{2n+1} - \sqrt{\bar{\epsilon}}^{\, 2n-1}}{\sqrt{\epsilon}^3-\sqrt{\bar{\epsilon}}}= \frac{\epsilon^n-\bar{\epsilon}^n}{\epsilon-\bar{\epsilon}},\]
and so $x_2(2n+1)=u_n$, which gives the desired formula for $x_2(k)$. \\

Next, since $x_3(1)=x_3(3)=0$, then using the recurrence for $x_3(k)$ in (\ref{recurrenceagain}), we see that $x_3(2n+1)=0$ for every $k \in \Z_{\geq 0}$. From (\ref{explicit}), we have
\[x_3(2n)=(\gamma_{31}+\gamma_{32}) \epsilon^n + (\gamma_{33}+\gamma_{34}) \bar{\epsilon}^{\, n}.\]
Since $x_3(2)=0$, we get
\[\gamma_{33}+\gamma_{34}=-(\gamma_{31}-\gamma_{32}) \epsilon^2\]
and so
$x_3(2n)=(\gamma_1+\gamma_2)(\epsilon^n - \bar{\epsilon}^{\,n-2}).$
Since $x_3(0)=1$ and $x_3(2)=0$, we have $x_3(4)=-1$, and so
\[\gamma_{31}+\gamma_{32}=\frac{-1}{\epsilon^2-1}.\]
So, as long as $n \geq 1$, we have
\[x_3(2n)=- \frac{\epsilon^n-\bar{\epsilon}^{\, n-2}}{\epsilon^2-1}=- \frac{\epsilon^{n-1} - \bar{\epsilon}^{\, n-1}}{\epsilon-\bar{\epsilon}}\]
and so $x_2(2n)=u_{n-1}$, which gives the desired formula for $x_2(k)$. The formula for $x_4(k)$ follows similarly to $x_2(k)$.
\end{proof}
\end{cor}

\begin{rmk} Let $M$ be an arbitrary full module in our quartic field $K$ and let $\alpha(k)=\beta \eta^k$ as above. Note that $M'=\beta \Z[\eta]$ is a finite index submodule of $M$ containing $\alpha(k)$ for every $k \in \Z_{\geq 0}$. So, we can always write the coordinate sequences for $\alpha(k)$ in terms of the basis constructed in Theorem \ref{thm1}. It turns out to be more challenging to apply Proposition \ref{prop} to find a basis for the entire module $M$. The following Proposition provides a characterization for when this can be done. \end{rmk}

First, we set some notation. For a basis $\{t_1, t_2, t_3, t_4\}$ of $M$, write
\begin{equation}
\label{B}
\begin{pmatrix} \beta \\ \beta \eta \\ \beta \eta^2 \\ \beta\eta^3 \end{pmatrix} = B \begin{pmatrix} t_1 \\ t_2 \\ t_3 \\ t_4 \end{pmatrix}.
\end{equation}
Note that $\exists X, Y \in \GL_4(\Z)$ so that $XBY=\diag(\delta_1, \dots, \delta_4)$ with $\delta_1 \mid \cdots \mid \delta_4$. Let $X=(x_{ij})$. Then, we have the following.

\begin{prop}\label{quarticprop} There is a choice of basis $W$ for the module $M$ so that the coordinate sequence $x_1(k)$ of $\alpha(k)$ with respect to the basis $W$ satisfies the initial conditions of Proposition \ref{prop} if and only if 
\[\gcd\left(\chi_4, \frac{\delta_4}{\delta_1}\right)=1,\]
where $\chi_i=x_{i2}+x_{i3}+(T+1)x_{i4}$, for $T = \epsilon + \bar{\epsilon}$. 

\begin{proof} Suppose that we have a basis $W=\{w_1, w_2, w_3, w_4\}$ for $M$ as above. Set
\[\vec{w}=\begin{pmatrix} w_1 & \cdots & w_4 \end{pmatrix}^{\top} \text{and }\, \vec{t}= \begin{pmatrix} t_1 & \cdots & t_4 \end{pmatrix}^{\top}.\]
Then, $A \vec{w}=B\vec{t}$, where $B$ is defined in (\ref{B}) and $A$ is a matrix with first column
\[\begin{pmatrix} 0 & a & a & a (T+1) \end{pmatrix}^{\top}.\] 
Write $D=\diag(\delta_1, \dots, \delta_4)$. Then, $D^{-1}XA \vec{w}=Y^{-1}\vec{t}$. Since $Y \in \GL_4(\Z)$, and $\vec{w}$ is a basis for $M$, we must have $C:=D^{-1}XA \in \GL_4(\Z)$. Observe that the first column of $C$ is of the form
\[ \begin{pmatrix} \frac{a}{\delta_1} \chi_1 & \frac{a}{\delta_2} \chi_2 & \frac{a}{\delta_3} \chi_3 & \frac{a}{\delta_4} \chi_4 \end{pmatrix}^{\top}.\]
Since $C \in \GL_4(\Z)$ the entries of this column must be relatively prime. In particular, this implies $a=\delta_4$ and $\gcd(\chi_4, \delta_4/\delta_1)=1$. Conversely, suppose we have $\gcd(\chi_4, \delta_4/\delta_1)=1$. Observe that $\gcd(\chi_1, \dots, \chi_4)=1$, since if there were a prime $p$ dividing every $\chi_i$, then we would have 
\[ p \cdot \begin{pmatrix} q_1 \\ \vdots \\ q_4 \end{pmatrix} = 0 \cdot \begin{pmatrix} x_{11} \\ \vdots \\ x_{41} \end{pmatrix}+
\begin{pmatrix} x_{12} \\ \vdots \\ x_{42} \end{pmatrix}+
 \begin{pmatrix} x_{13} \\ \vdots \\ x_{43} \end{pmatrix}+
(T+1) \begin{pmatrix} x_{14} \\ \vdots \\ x_{44} \end{pmatrix},\]
where $q_i \in \Z$. But then the columns of $X$ would be $(\Z/p)$-linearly dependent, which contradicts the fact that $X \in \GL_4(\Z)$. Now, let 
\[\bar{c}_1 = \begin{pmatrix} \frac{\delta_4}{\delta_1} \chi_1 & \frac{\delta_4}{\delta_2}\chi_2 & \frac{\delta_4}{\delta_3}\chi_3 & \chi_4 \end{pmatrix}^{\top}.\]
A standard result in Geometry of Numbers tells us that a lattice element can be lifted to a basis precisely when it is primitive (see Chapter 1 of \cite{cassels}, for example). Since $\delta_1 \mid \cdots \mid \delta_4$, and we've assumed that $\gcd(\chi_4, \delta_4/\delta_1)=1$, then we have
\[\gcd\left(\frac{\delta_4}{\delta_1} \chi_1, \frac{\delta_4}{\delta_2}\chi_2, \frac{\delta_3}{\delta_2} \chi_3, \chi_4\right)=1.\]
So, there is a matrix $C \in \GL_4(\Z)$ with first column equal to $\vec{c}_1$.
Next, let $A=X^{-1} D C$. Then, $A$ has first column
\[\vec{a}_1=\begin{pmatrix} 0 & \delta_4 & \delta_4 & \delta_4 (T+1) \end{pmatrix}^{\top}.\]
Furthermore, $XD^{-1}A =C \in \GL_4(\Z)$. Let $Z=YD^{-1} X A \in \GL_4(\Z)$, and define a new basis $W=\{w_1, w_2, w_3, w_4\}$ from $\{t_1, t_2, t_3, t_4\}$ by change of basis matrix $Z^{-1}$. 
Since $Z=B^{-1} A$, we have
\[A \begin{pmatrix} w_1 \\ \vdots \\ w_4 \end{pmatrix} = \begin{pmatrix} \beta \\ \vdots \\ \beta \eta^3 \end{pmatrix}.\]
So, if we write $\alpha(k)=x_1(k)w_1+\cdots+x_4(k)w_4$, then $x_1(k)$ satisfies the initial conditions
$x_1(0)=0, \, x_1(1)=x_1(2)=\delta_4, \, x_1(3)=\delta_4(T+1).$ 
\end{proof}
\end{prop}

Theorem \ref{thm2} provides a family of modules satisfying the conditions of Proposition \ref{quarticprop}. An interesting future direction could be to provide a characterization of all such modules.

\subsection*{Proof of Theorem \ref{thm2}} Recall that $M=\Z[\sqrt{m}, \sqrt{n}]$ with $m=n+1$, and $\eta=\sqrt{m}+\sqrt{n}$. 
Observe that $\eta=\sqrt{\epsilon}$, where $\epsilon=m+n+2\sqrt{mn}$. Let $K=\Q(\sqrt{m}, \sqrt{n})$ with $m, n$ as above, and $L=\Q(\sqrt{mn})$.  A short computation shows that $N_L(\epsilon)=1$ and $\eta \in \U_M^+$. Next, observe that
\[\begin{pmatrix} 1 \\ \eta \\ \eta^2 \\ \eta^3 \end{pmatrix}=\underbrace{\begin{pmatrix} 1 & 0 & 0 & 0 \\ 0 & 1 & 1 & 0 \\ 2m+1 & 0 & 0 & 2 \\ 0 & 4m+3 & 4m+1 & 0 \end{pmatrix}}_{B} \begin{pmatrix} 1 \\ \sqrt{m} \\ \sqrt{n} \\ \sqrt{mn} \end{pmatrix}.\]
We can compute $XBY=\diag(1, 1, 2, 2)$ where
\[X=\begin{pmatrix} 1 & 0 & 0 & 0 \\ 0 & 1 & 0 & 0 \\ 0 & -4m-1 & 0 & 1 \\ -2m-1 & 0 & 1 & 0 \end{pmatrix},
\text{ and }
Y= \begin{pmatrix} 1 & 0 & 0 & 0 \\ 0 & 0 & 1 & 0 \\ 0 & 1 & -1 & 0 \\ 0 & 0 & 0 & 1 \end{pmatrix}.
\]
Hence, $\chi_4=1$ and so Proposition \ref{quarticprop} applies. That is, there is a basis $W$ so that the coordinate sequence $x_1(k)$ of $\alpha(k)$ with respect to the basis $W$ satisfies the initial conditions of Proposition \ref{prop}. So, $x_1(k)$ is a LDS.  \hfill \qed 

\begin{rmk} Note that the proof of Proposition \ref{quarticprop} provides an algorithm for computing our desired basis in Theorem \ref{thm2} explicitly. We demonstrate this computation. Note that \[\Tr_{K/ \Q}(\epsilon)=2(m+n)=2m+2,\] where we've used the assumption that $n=m+1$. So, we need to find a matrix $C \in \GL_4(\Z)$ with first column
$\vec{c}_1= \begin{pmatrix} 0 & 2 & 4(1-m) & 1 \end{pmatrix}^{\top}.$
For example, we can take
\[C=\begin{pmatrix} 0 & 0 & 1 & 0 \\ 2 & 0 & 0 & 1 \\ 2(1-m) & 1 & 0 & 0\\ 1 & 0 & 0 & 0 \end{pmatrix}. \]
Then, we compute $A=X^{-1} D C$, where $D=\diag(1, 1, 2, 2)$, to get
\[A=\begin{pmatrix} 0 & 0 & 1 & 0 \\ 2 & 0 & 0 & 1 \\ 2 & 0 & 2m+1 & 0 \\ 2(2m+3) & 2 & 0 & 4m+1 \end{pmatrix}.\]
So, setting $Z=B^{-1} A$, and using $Z^{-1}$ as our change of basis matrix from $\{1, \sqrt{m}, \sqrt{n}, \sqrt{mn}\}$ we obtain basis $W=\{w_1, w_2, w_3, w_4\}$ for $M$ given by
\[w_1=\sqrt{mn}, \,\, w_2=\sqrt{m}+2(m-1)\sqrt{mn},\]
\[w_3=1, \,\, w_4=\sqrt{m}+\sqrt{n}-2\sqrt{mn}.\]
So, if we write $\eta^k=x_1(k)w_1+\cdots+x_4(k)w_4$, we can check that $x_1(k)$ satisfies the initial conditions
$x_1(0)=0, \,\, x_1(1)=x_1(2)=2, \,\, x_1(3)=2(2m+3),$
and so by Proposition \ref{prop} we have that $x_1(k)$ is a LDS.
\end{rmk}

\begin{rmk} If $\alpha(k)$ is as in Theorems \ref{thm1} and \ref{thm2}, the coordinate sequences $\{x_i(k): k \in \Z_{\geq 0}\}$ contain order two subsequences $\{x_i(2k): k \in \Z_{\geq 0}\}$. So, as discussed in Section 3, it is not possible to find a basis for the corresponding module that makes $x_1(k), x_2(k), x_3(k), x_4(k)$ LDS simultaneously. 
\end{rmk}


\section{Powers of Algebraic Integers}
\label{sec: 5}

We conclude this paper by discussing a related sequence studied by Silverman in \cite{algints}, and show how methods from the previous sections might be used in its analysis.\\

Given $\alpha \in \bar{\Z}$, define the sequence
\begin{equation}
\label{dkdef}
d_k(\alpha) = \max\{ d \in \Z \mid \alpha^k \equiv 1 (\mod d)\}.
\end{equation}
where the congruence $\alpha^k \equiv 1 (\mod d)$ means that there is an element $\beta \in \bar{\Z}$ with  \[\alpha^k = 1+d\beta.\] 
 In \cite{algints}, Silverman proved that $d_k(\alpha)$ is a divisibility sequence, and showed that, except for some exceptional cases, this sequence grows slower than exponentially. We record this Theorem below. 

\begin{thm}[\cite{algints} Theorem 1] \label{growth} Let $\alpha \in \bar{\Z}$. Then, 
\[\lim_{n \to \infty} \frac{\log d_n(\alpha)}{n} = 0\]
unless $\alpha^\ell \in \Z$ for some $\ell \in \Z$, or $\alpha^\ell$ is a unit in a quadratic extension of $\Q$.
\end{thm}

In the exceptional case where $K = \Q(\alpha)$ is a quadratic number field with $N_K(\alpha) =1$, it is shown in Theorem 7 of \cite{algints} that $d_k: = d_k(\alpha)$ satisfies the order 4 linear recurrence
\begin{equation} \label{dkrecurrence}
d_{k+4} = T d_{k+2}- d_k,
\end{equation}
where $T=\alpha + \bar{\alpha}$.\\

Suppose that $\alpha$ is an element in a real quadratic field and let $\eta$ be an algebraic number with $\eta^2 = \alpha$. Choose an integral basis $\{w_1, w_2, w_3, w_4\}$ for $L = \Q(\eta)$, and write
\begin{equation}
\label{xidef5millionthtime} \eta^k = x_1(k)w_1+\cdots+x_4(k)w_4.
\end{equation}
Then, by Corollary \ref{cor1} we have that the $x_i(k)$ also satisfy the order 4 linear recurrence 
\begin{equation}
\label{xirecurrence}
x_i(k+4)=Tx_i(k+2)-x_i(k).
\end{equation}
We have the following observation. 

\begin{prop} \label{c} Suppose that $\alpha$ is a unit of positive norm in a real quadratic field, and let $\eta$ be any algebraic integer satisfying $\eta ^2 = \alpha$. Let $L = \Q(\eta)$. 
If $\O_L = \Z[\eta]$, then there exists a choice of basis for $\O_L$ so that
\[
x_1(k)=\begin{cases} d_k(\alpha) & \text{ if } d_1(\alpha) = 1 \\ d_k(\alpha)/d_1(\alpha) & \text{ if } d_1(\alpha) \not=1,\end{cases}
\]
where $d_k(\alpha)$ and $x_1(k)$ are defined as in (\ref{dkdef}) and (\ref{xidef5millionthtime}), respectively. 

\begin{proof} By (\ref{dkrecurrence}) and (\ref{xirecurrence}), we know that $x_1(k)$ and $d_k(\alpha)$ both satisfy the same recurrence. So, it suffices to find a basis for $\O_L$ so that the initial conditions of $x_1(k)$ match those of $d_k(\alpha)$. Let $\vec{a} = \begin{pmatrix} d_0 & d_1 & d_2 & d_3 \end{pmatrix}$. If $d_1(\alpha)=1$, then $\vec{a}$ is a primitive lattice element in $\Z^4$ and so $\exists A \in \GL_4(\Z)$ with first column $\vec{a}$. Let $\{w_1, w_2, w_3, w_4\}$ be the basis obtained from $\{1, \sqrt{\alpha}, \sqrt{\alpha}^2, \sqrt{\alpha}^3\}$ by change of basis matrix $A^{-1}$. Then $x_1(k)$ has the desired initial conditions. If $d_1(\alpha) \not=1$, then we replace the sequence $d_k(\alpha)$ by $d_k(\alpha)/d_1(\alpha)$ in the argument above. 
\end{proof}
\end{prop}

\begin{rmk} Since the coordinate sequence $\{x_i(k): k \in \Z_{\geq 0}\}$ of $\beta^k$ for any $\beta \in \bar{\Z}$ defined in (\ref{xidef5millionthtime}) are linear with distinct characteristic roots, they grow exponentially (see Chapter 2 of \cite{recurrence}, for example). So Theorem \ref{growth} implies that if $x_i(k) = d_k(\alpha)$ for some $\alpha \in \bar{\Z}$ and fixed index $i$, then $\alpha$ must be in one of the exceptional cases. That is, we must have a power of $\alpha$ either in $\Z$ or a quadratic unit. It would be interesting to know when a results like Proposition \ref{c} holds in the other exceptional cases. 
\end{rmk}

We finish this section by discussing how the recurrence for the coordinate sequences $\{x_i(k): k \in \Z_{\geq 0}\}$ of $\alpha^k$, for some $\alpha \in \bar{\Z}$, could be used to study the sequence $d_k(\alpha)$ defined in (\ref{dkdef}). \\

By definition, we have that $d_k(\alpha)$ is the largest positive integer satisfying 
\begin{equation}
\label{dkalternate}
\alpha^k - 1 \in d_k(\alpha) \O_K.
\end{equation}
 Let $\{1, w_2, \dots, w_n\}$ be an integral basis for for $K$. If we write
\[\alpha^k = x_1(k)+x_2(k)w_2 + \cdots + x_n(k)w_n.\]
Then we have 
\begin{equation}
\label{dkxi} d_k(\alpha) = \gcd (x_1(k)-1, x_2(k), \dots, x_n(k)).
\end{equation}
Furthermore, by Corollary \ref{cor1} each of the sequences $x_i(k)$ has characteristic polynomial equal to the minimal polynomial of $\alpha$. As in the previous sections, we can change the initial conditions of $x_i(k)$ by changing the basis of $\O_K$. We use these observations to study the following conjecture.

\begin{conj}[\cite{algints} Conjecture 9] \label{silvermanconj}For $\alpha \in \bar{\Z}$, suppose one of the following holds:
\begin{enumerate}
\item[(a)]  $[\Q(\alpha^r): \Q] \geq 3 \,\, \forall r \geq 1$, or
\item[(b)] $[\Q(\alpha^r): \Q] \geq 2 \,\, \forall r \geq 1 \text{ and } N_K(\alpha) \not=\pm 1.$
\end{enumerate}
Then, the set $\{ k \geq 1 \mid d_k(\alpha) = d_1 (\alpha)\}$ is infinite
\end{conj}

The following proposition can be used to construct examples where Conjecture \ref{silvermanconj} holds. 

\begin{prop} \label{lastprop} Let $\alpha \in \bar{Z}$ have minimal polynomial
\[f(X)=X^n-s_1X^{n-1} - \cdots - s_n,\]
and set $K = \Q(\alpha)$. If there exists a positive divisor $t$ of $n$ so that $s_i =0$  $\forall i \not\in t\Z$, then we have
$d_n(\alpha) \leq \disc(1, \alpha, \dots, \alpha^{n-1})$
for all $n \in 1+ t \Z$. In particular, if $\O_K = \Z[\alpha]$, then $d_n(\alpha)=d_1(\alpha)=1$ for all $n \in 1+ t \Z$. 

\begin{proof} Let $\Delta = \disc(1, \alpha, \dots, \alpha^{n-1})$, and $\tilde{d}_n$ denote the largest integer with 
\[\alpha^k-1 \in \tilde{d}_k \Z[\alpha].\]
Since $\O_K \subset \frac{1}{\Delta} \Z[\alpha]$, then by (\ref{dkalternate}) we have
\[ \alpha^k - 1 \in d_k \O_K \subset \frac{d_k}{\Delta} \Z[\alpha].\]
By definition of $\tilde{d}_k$, we have $d_k \leq \Delta \tilde{d}_k$. \\

Now, write
\[\alpha^k = y_1(k)+y_2(k) \alpha +\cdots + y_n(k) \alpha^{n-1}.\]
Then, similar to (\ref{dkxi}) we observe that 
\[\tilde{d}_k = \gcd(y_1(k) -1, y_2(k), \dots, y_n(k)).\] If $s_i=0$ $\forall i \not\in t \Z$, then by Corollary \ref{cor1} we have
\[y_1(k+n)= s_t y_1(k+n-t) + \cdots + s_{\ell t} y_i (k+n - \ell t). \]
Since $y_1(k)$ has initial conditions $y_1(0)=1, y_1(1)=\cdots = y_1(n-1) =0$, we see that $y_1(1 + \ell t) =  0$ for any $\ell \in \Z_{ \geq 0}$. 
\end{proof}
\end{prop}

We give an example to demonstrate how to use Proposition \ref{lastprop} to generate examples where Conjecture \ref{silvermanconj} holds.

\begin{example*} Let $\beta$ be an element of a quadratic number field, and $\alpha$ any algebraic integer satisfying $\alpha^2 = \beta$. Observe that the minimal polynomial of $\alpha$ is of the form
\[f(X)=X^4-\Tr_{L/ \Q}(\beta) X^2 + N_{L/\Q}(\beta),\]
where $L = \Q(\beta)$, 
and so by Proposition \label{lastprop} we have that
\[d_k(\alpha) \leq \disc(1, \alpha, \alpha^2, \alpha^3)\] 
whenever $k$ is odd. So, Conjecture \ref{silvermanconj} holds whenever 
\begin{equation}
\label{monogenic} \O_K=\Z[\alpha] \text{ where $K = \Q(\alpha)$. }
\end{equation}
Number fields $K$ satisfying (\ref{monogenic}) are called \textit{monogenic}. \\

We searched for examples of such $\alpha$ by letting $\alpha$ be a quartic unit satisfying $\alpha^2 = : \epsilon$, for $\epsilon$ a fundamental unit in a quadratic subfield of $K$, as in Section \ref{sec: 4}. Using Sage to check whether $\O_K = \Z[\alpha]$, we checked the first ten real quadratic fields (ordered by discriminant), and found the following list of examples:
\[\textstyle \sqrt{1+\sqrt{2}}, \sqrt{\frac{1}{2}(1+\sqrt{5})}, \sqrt{3+\sqrt{10}}, \sqrt{\frac{1}{2}(3+\sqrt{13})}, \sqrt{4+\sqrt{15}}\]
It would be interesting to provide a characterization of all monogenic biquadratic fields with generator of the form $\alpha$ where $\alpha^2$ is a unit in a quadratic subfield, as this would provide a family of examples where Conjecture \ref{silvermanconj} holds for the sequence $d_k(\alpha)$. We note that the problem of characterizing monogenic biquadratic fields was studied in \cite{gaal}.  
\end{example*}

\section*{Acknowledgments}

This project was initiated during my visit to the Max Planck Institute for Mathematics, in Bonn, in Spring 2019. I acknowledge the support provided by MPIM in the early stages of this project. This project was also partially supported by the National Science Foundation award DMS-2001281. I would also like to thank Professor Shabnam Akhtari for her support in this work, as well as Professor Joseph Silverman for comments and suggestions on an earlier version of this manuscript, in particular for pointing out the related sequences discussed in Section 5.


\vspace{4em}

\small
Department of Mathematics, University of Oregon \\
\textit{E-mail address: }{\href{mailto:ebellah@uoregon.edu}{ebellah@uoregon.edu}} \\

\end{document}